\documentclass[12pt]{article}

\usepackage{amssymb}%
\usepackage{amsmath}%
\usepackage{amsthm}%

\usepackage{url}%

\usepackage{fullpage}%
\usepackage{xcolor}%

\allowdisplaybreaks%

{\theoremstyle{plain}%
 \newtheorem{theorem}{Theorem}

 \newtheorem{lemma}{Lemma}%
}
{\theoremstyle{remark}

\newtheorem{remark}{Remark}
}
{\theoremstyle{definition}

\newtheorem{example}{Example}
}
\newcommand{\cl}{{\rm Cl}}
\newcommand{\Ls}{{\rm Ls}}

\newcommand\nc{\newcommand}

\nc\gl{{\lambda}}

\nc{\bfs}{{\boldsymbol{\sl{s}}}}

\nc{\binn}{{\binom{2n}{n}}}

\nc\tx{{\texttt{x}}}

\nc\td{{\texttt{d}}}

\nc\ta{{\texttt{a}}}

\nc\ty{{\texttt{y}}}

\makeatletter
\newcommand{\dotDelta}{{\vphantom{\Delta}\mathpalette\d@tD@lta\relax}}
\newcommand{\d@tD@lta}[2]{%
 \ooalign{\hidewidth$\m@th#1\mkern-1mu\cdot$\hidewidth\cr$\m@th#1\Delta$\cr}%
}
\makeatother

\begin{document}

\begin{center}
{\Large On a problem involving the squares of odd harmonic numbers}
\end{center}

\begin{center}

John M. Campbell\footnote{Corresponding author: {\tt jmaxwellcampbell@gmail.com}.  ORCID: 0000-0001-5550-2938}

{\footnotesize Department of Mathematics and Statistics, York University, 4700 Keele St, Toronto, Ontario, Canada M3J~1P3 }

\vspace{0.1in}

Paul Levrie

{\footnotesize Faculty of
 Applied Engineering, UAntwerpen, Groenenborgerlaan 171, 2020 Antwerp, Belgium, and Department of Computer Science, KU Leuven,
 Celestijnenlaan 200A, 3001 Heverlee (Leuven), Belgium, P.~O. Box 2402}

\vspace{0.1in}

Ce Xu

{\footnotesize School of Mathematics and Statistics, Anhui Normal University, Wuhu 241002, People's Republic of China}

\vspace{0.1in}

Jianqiang Zhao

{\footnotesize Department of Mathematics, The Bishop’s School, La Jolla, CA 92037, USA}

 \

\end{center}

\begin{abstract}
	\noindent
 We introduce a full solution to a problem considered by Wang and Chu concerning series involving the squares of finite sums of the form $1 + \frac{1}{3}
 + \cdots + \frac{1}{2n-1}$. Our proof involves techniques from the theory of colored multiple zeta values.
\end{abstract}



\section{Introduction}
 As Adamchik has expressed \cite{Adamchik2002} (cf.\ \cite{CampbellChen2022}), there is a great history concerning the study of series of the following form:
\begin{equation}\label{Sfunction}
 S(r) = \sum_{k=0}^{\infty} \left( \frac{1}{16} \right)^{k} \binom{2k}{k}^2 \frac{ 1 }{ k + r }.
\end{equation}
 As in \cite{Adamchik2002,CampbellChen2022}, we recall that Ramanujan was interested in the $S$-function shown above \cite[p.\ 39]{Berndt1989}, and we
 record Ramanujan's identity whereby
 \begin{equation}\label{SfunctionRamanujan}
 S(r) = \frac{16^{r}}{ \pi r^2 \binom{2r}{r}^2 } \sum_{k=0}^{r - 1} \left( \frac{1}{16} \right)^{k} \binom{2k}{k}^2
 \end{equation}
 for $r \in \mathbb{N}_{0}$. Ramanujan's discoveries concerning the $S$-function in \eqref{Sfunction}, together with the large amount of research, over the
 years, concerning this function \cite{Adamchik2002} \cite[pp.\ 39--40]{Berndt1989}, 
 have inspired the development of \emph{Ramanujan-like} series given by introducing summand
 factors, such as harmonic or harmonic-type numbers, in the series shown in \eqref{Sfunction} \cite{Campbell2018,WangChu2020}. In this article, we solve a
 problem recently considered by Wang and Chu \cite{WangChu2022} concerning a family of series resembling \eqref{Sfunction} and involving harmonic-type
 numbers, and we introduce many related evaluations.

 We write $O_{n} = 1 + \frac{1}{3} + \cdots + \frac{1}{2 n - 1}$ to denote the $n^{\text{th}}$ odd harmonic number, by analogy with the usual notation
 $H_{n} = 1 + \frac{1}{2} + \cdots + \frac{1}{n}$ for the $n^{\text{th}}$ entry in the classical sequence of harmonic numbers. In the recent article
 \cite{WangChu2022}, Wang and Chu left the problem of evaluating series of the form
\begin{equation} \label{mainproblem}
 \sum_{n=1}^{\infty} \left( \frac{1}{16} \right)^{n} \binom{2n}{n}^2 \frac{O_{n}^2 }{(1 + 2n - 2 \lambda)^2}
\end{equation}
 as an open problem, letting $\lambda \in \mathbb{N}$. Even for the base case such that $\lambda = 1$, evaluating the series in \eqref{mainproblem} 
 is quite difficult, and it appears that the coefficient-extraction methods employed by 
 Chu et al.\ \cite{Chu2023,WangChu2020,WangChu2022} do not apply to
 \eqref{mainproblem}. We offer a complete solution to the problem of evaluating \eqref{mainproblem} for an arbitrary parameter $\lambda \in \mathbb{N}$.

\section{The base case}\label{sectionbase}
 We adopt notation from \cite{CampbellLevrieNimbran2021}, writing
\begin{equation}\label{Catalanlike}
 \mathcal{G} := \Im\left( \text{Li}_{3}\left( \frac{i+1}{2} \right) \right)
\end{equation}
 to denote the Catalan-like constant explored in \cite{CampbellLevrieNimbran2021}, letting $ G:= \sum_{n=1}^{\infty} \frac{(-1)^{n+1}}{(2n-1)^2} $
 denote the ``original'' Catalan constant, and writing $\text{Li}_{s}(z) = z + \frac{z^{2}}{2^{s}} + \frac{z^{3}}{3^{s}} + \cdots$ to denote the
 polylogarithm function. {\color{black} As discussed in \cite{CampbellLevrieNimbran2021},
 the trilogarithmic expression in \eqref{Catalanlike}
 has been considered as an irreducible constant in a variety of different applications and by a variety of different authors,
 and is a naturally occuring, higher-order analogue of
 both $G = \Im(\text{Li}_{2}(i))$ and
 the constant $\mathfrak{G} = \Im(\text{Li}_{2}(\sqrt[3]{-1}))$
 known as \emph{Gieseking's constant} \cite{Adams1998}.
 The uses
 of the irreducible constant in \eqref{Catalanlike} related to both quantum field theory
 \cite{Coffey2008} and the study of Euler-like sums \cite{CantariniDAurizio2019,SofoNimbran2020}
 further motivate Theorem \ref{maintheorem} below,
 which is the base case for our full solution to Wang and Chu's problem concerning \eqref{mainproblem}.
 Our full solution is highlighted as Theorem \ref{fullsolution} below. }

\begin{theorem}\label{maintheorem}
 The infinite series
 $$ \sum_{k=1}^{\infty} \left( \frac{1}{16} \right)^{k} \binom{2k}{k}^2
 \frac{O_{k}^{2}}{(2k-1)^2} $$
 admits the symbolic form
\begin{equation*}
 -\frac{8 G}{\pi }-\frac{8 G \ln (2)}{\pi }-\frac{48 \mathcal{G}}{\pi }+\frac{9 \pi
 ^2}{8}+\frac{\pi }{6}+\frac{4 \ln ^2(2)}{\pi }+\frac{3 \ln ^2(2)}{2}.
\end{equation*}
\end{theorem}

\begin{proof}
 Using the Wilf--Zeilberger method \cite{PetkovsekWilfZeilberger1996}, Campbell \cite{Campbell2022} had recently proved the following equality, noting the
 appearance of a copy of the base case of \eqref{mainproblem} for $\lambda = 1$:
\begin{align}
 & 4 \sum _{k=0}^{\infty} \left( \frac{1}{16} \right)^{k} \binom{2k}{k}^2\frac{O_{k}}{2 k+1}
 - \sum _{k=0}^{\infty} \left( \frac{1}{16} \right)^{k} \binom{2k}{k}^2
 \frac{O_{k}^2}{2 k-1}-2 \sum _{k = 0}^{\infty} \left( \frac{1}{16} \right)^{k} \binom{2k}{k}^2 \frac{O_{k}^2}{(2 k-1)^2}
 \nonumber \\
 & \hspace{1cm} = \frac{12 G}{\pi }+\frac{32 \mathcal{G} }{\pi } - \frac{6 \ln ^2(2)}{\pi }
 - \frac{3 \pi ^2}{4}-\frac{\pi }{4} - \ln ^2(2). \label{wantcentral}
\end{align}
 Using a recursive proof approach, Campbell \cite{Campbell2022} had also evaluated the ``central'' series in \eqref{wantcentral}:
\begin{equation}\label{Basel}
 \sum _{k=0}^{\infty} \left( \frac{1}{16} \right)^{k} \binom{2k}{k}^2
 \frac{O_{k}^2}{2 k-1} = \frac{4G}{\pi} - \frac{\pi}{12} - \frac{2\ln^2 (2)}{\pi},
 \end{equation} 	
 giving us that the following equality holds:
\begin{align}
 & 4 \sum _{k=0}^{\infty}\left( \frac{1}{16} \right)^{k} \binom{2k}{k}^2
 \frac{O_{k} }{2 k+1} -
 2 \sum _{k=0}^{\infty} \left( \frac{1}{16} \right)^{k} \binom{2k}{k}^2\frac{
 O_{k}^{2} }{(2 k-1)^2} \label{nocentral} \\
 & = \frac{16 G}{\pi }+\frac{32 \mathcal{G} }{\pi }-\frac{3 \pi ^2}{4}-\frac{\pi }{3}-\frac{8 \ln ^2(2)}{\pi }-\ln ^2(2). \nonumber
\end{align}
 So, the problem of evaluating the base case for \eqref{mainproblem} is equivalent to the problem of evaluating $$ \sum _{k=0}^{\infty} \left( \frac{1}{16} \right)^{k} \binom{2k}{k}^2
 \frac{ O_{k} }{2 k+1}. $$ In this direction, we claim that the following equality holds:
\begin{equation}\label{Paulbrilliant}
 \sum _{k=0}^{\infty} \left( \frac{1}{16} \right)^{k} \binom{2k}{k}^2
 \frac{2 O_{k} - H_k}{2 k+1}
 = -\frac{16 \mathcal{G} }{\pi }+\frac{\pi ^2}{8}+\frac{8 G \ln (2)}{\pi }+\frac{\ln ^2(2)}{2}.
\end{equation}
 To prove this, we need the following integrals:
\begin{align}
& \int_{0}^{\pi/4} (\ln (\sin(t)) - \ln(\cos (t)) ) dt = - G, \label{eq:1} \\
 & \int_{0}^{\pi/4} (\ln^2 (\sin(t)) - \ln^2 (\cos (t)) ) dt = \frac{5\pi^3}{64}+\frac{1}{16}\pi \ln^2(2)-2\mathcal{G}+ G \ln(2). \label{eq:2}
\end{align}
 The first is a well known expression for Catalan's constant, whereas the second one can be found in \cite{CampbellLevrieNimbran2021}.
 We are to also use the following result \cite{NimbranLevrie2022}:
\begin{equation}\label{eq:7}
 \int_0^{\arcsin(x)}\ln^{n}(\sin u)\, {du} =
 \sum_{i=0}^{n} \left[(-1)^i i! \binom{n}{i} \ln^{n-i} (x) \sum_{k =
 0}^\infty \frac{\binom{2k}{k}\, x^{2k+1}}{(2k+1)^{i+1} \,2^{2k}}\right].
\end{equation}
 For $n=1$, \eqref{eq:7} reduces to: $$\int_0^{\arcsin(x)}\ln(\sin u)\, {du} = \ln (x) \sum_{k=0}^\infty \frac{\binom{2k}{k}\, x^{2k+1}}{(2 k +
 1) \,2^{2k}} - \sum_{k=0}^\infty \frac{\binom{2k}{k}\, x^{2k+1}}{(2k+1)^{2} \,2^{2k}}. $$
 We divide by $x$, replace $x$ by $\sin v$, and integrate the result between $0$ and $\tfrac{\pi}{2}$:
 \begin{align} 
 & \int_0^{\pi/2} \left(\frac{1}{\sin v} \int_0^v \ln(\sin u)\, {du} \right){dv} = \nonumber \\ 
 & \sum_{k=0}^\infty \frac{\binom{2k}{k}}{(2k+1) \,2^{2k}} \int_0^{\pi/2} \ln (\sin v) \sin^{2k} v \, dv 
 	- \frac{\pi}{2}\sum_{k=0}^\infty \frac{\binom{2k}{k}^2}{(2k+1)^{2} \,2^{4k}}. \label{eq:8} 
\end{align} 
 In the double integral on the left-hand side, we change the order of integration: 
\begin{align*}
\int_0^{\pi/2} \left(\frac{1}{\sin v} \int_0^v \ln(\sin u)\, {du} \right){dv} & = \int_0^{\pi/2} \left(\ln(\sin u) \int_u^{\pi/2} \frac{1}{\sin v}\, {dv} \right) {du} \\
& = -\int_0^{\pi/2} \ln(\sin u) \ln (\tan \tfrac{u}{2}) {du}.
\end{align*}
Using the substitution $u=2t$ the integral can be rewritten as:
\begin{align}
&-2 \int_0^{\pi/4} \left[ \ln 2 \cdot (\ln (\sin(t)) - \ln(\cos (t)) ) + \ln^2 (\sin(t)) - \ln^2 (\cos (t)) \right] {dt} \label{eq:9}\\
& \hspace{3cm}= -\frac{5}{32} \pi^3 - \frac{1}{8}\pi \ln^2 2 + 4 \mathcal{G} \nonumber
\end{align}
using \eqref{eq:1} and \eqref{eq:2}. The integral in 
 the right-hand side of \eqref{eq:8} is given in \cite[4.241.1]{GradshteynRyzhik2000}. 
 Explicitly, writing 
 $I_k:=\int_0^{\pi/2} \ln (\sin v) \sin^{2k} v \, dv$, we may 
 evaluate $I_{k}$ so that 
 $I_k = \frac{\pi}{2} \frac{1}{2^{2k}} \binom{2k}{k} (O_k - \frac{1}{2}H_k - \ln 2)$. 
 Hence, the right-hand side of \eqref{eq:8} reduces to: 
 $$\frac{\pi}{2}\sum_{k=0}^\infty \left( \frac{1}{16} \right)^{k} \binom{2k}{k}^2 (O_k - \frac{1}{2}H_k - \ln 2) \frac{1}{2k+1}
- \frac{\pi}{2}\sum_{k=0}^\infty \left( \frac{1}{16} \right)^{k} \binom{2k}{k}^2 \frac{1}{(2k+1)^{2} }.$$
 Multiplying both sides of \eqref{eq:8} by $\frac{4}{\pi}$ and using \eqref{eq:1}, \eqref{eq:2} and \eqref{eq:9} we obtain
 an equivalent form of \eqref{Paulbrilliant}.

 So, it remains to evaluate the following series:
\begin{equation}\label{remainsHk}
 \sum _{k=0}^{\infty} \left( \frac{1}{16} \right)^{k} \binom{2k}{k}^2
 \frac{H_{k} }{2 k+1}.
\end{equation}
 Using a reindexing argument, we obtain the following:
\begin{align*}
 & \sum _{k=0}^{\infty} \left( \frac{1}{16} \right)^{k} \binom{2k}{k}^2
 \frac{H_{k} }{2 k+1} = \sum _{k = 1}^{\infty} \left( \frac{1}{16} \right)^{k-1} \binom{2(k-1)}{k-1}^2
 \frac{H_{k-1}}{2 k-1} \\
 & = -\frac{4 (2 G - 1)}{\pi } +
 \sum _{k = 1}^{\infty} \left( \frac{1}{16} \right)^{k}
\binom{2 k}{k}^2 H_k\left(\frac{1}{(2k-1)^3}+\frac{2}{(2k-1)^2}+\frac{1}{2k-1}\right).
\end{align*}
 Using previously known Ramanujan-like series introduced in \cite{Campbell2018,Chen2016}, we obtain that the series in \eqref{remainsHk} equals:
 $$ \frac{24}{\pi }-\frac{8 G}{\pi }-\frac{24 \ln (2)}{\pi } + \sum _{k=1}^{\infty} \left( \frac{1}{16} \right)^{k} \binom{2 k}{k}^2 \frac{ H_k }{(2
 k-1)^3}. $$
 We claim that
\begin{equation}\label{equ:useItIntegral}
 \sum _{k=1}^{\infty}
 \left( \frac{1}{16} \right)^{k} \binom{2 k}{k}^2
 \frac{ H_k }{(2 k-1)^3}
\end{equation}
 equals $$ \frac{8 G}{\pi }-\frac{16 G \ln (2)}{\pi }-\frac{16 \mathcal{G} }{\pi }+\frac{5
 \pi^2}{8}-\frac{24}{\pi }+\frac{\ln ^2(2)}{2}+\frac{24 \ln (2)}{\pi }. $$
 To show this, we use an interated integral-based approach, using
 results from \cite{XuZhao2022}.
 We recall that for any real numbers $a,b$ and $1$-forms $f_j(t)\, dt$ ($j=1,\dots, d$)
 with $d>1$ we may recursively define the iterated integral
\begin{equation*}
\int_a^b f_1(t)\, dt\, f_2(t)\, dt \cdots f_d(t)\, dt:=\int_a^b f_1(u) \Big(\int_a^u f_2(t)\, dt \cdots f_d(t)\, dt \Big)\, du.
\end{equation*}
By the proof of \cite[Theorem 6.1]{XuZhao2022} we may use iterated integrals to compute \eqref{equ:useItIntegral} as follows:
\begin{align*}
I_s:=&\,\sum_{k>m>0} \left( \frac{1}{16} \right)^{k} \binom{2k}{k}^2\frac{1}{(2k-1)^{s+2}(2m)}\\
=&\,\frac{2}{\pi} \int_0^{\pi/2} \sin t\, dt (\cot t\,dt)^{s}\, \cot^2 t\,dt\,d(\sec t) \Big(\csc t-\cot t\Big) \,dt\\
=&\,\frac{2}{\pi} \int_0^{\pi/2} \sin t\, dt (\cot t\,dt)^s\, (\csc^2 t-1)\,dt\,d(\sec t) \Big(\csc t-\cot t\Big) \,dt
= X_s-Y_s-I_{s-1},
\end{align*}
where
\begin{align*}
X_s:=&\,\frac{2}{\pi} \int_0^{\pi/2} \sin t\, dt (\cot t\,dt)^s\, \sec t\,dt\, \Big(\csc t-\cot t\Big) \,dt,\\
Y_s:=&\, \frac{2}{\pi} \int_0^{\pi/2} \sin t\, dt (\cot t\,dt)^s\,dt\,d(\sec t) \Big(\csc t-\cot t\Big) \,dt.
\end{align*}
Set 1-forms $\ta=dt/t$, $\tx_\mu=dt/(\mu-t)$, $\td_{\mu,\nu}=\tx_{\mu}-\tx_{\nu}$ for any roots of unity $\mu$ and $\nu$, and $\ty=\tx_{i}+\tx_{-i}-\tx_{1}-\tx_{-1}$. Then
\begin{align*}
X_0=&\,\frac{2}{\pi} \int_0^{\pi/2} dt\, \Big(\csc t-\cot t\Big) \,dt
= \frac{2}{\pi} \int_0^1 i \Big(2\tx_{-1}-\tx_{i}-\tx_{-i}\Big) \td_{-i,i}, \\
X_1=&\,\frac{2}{\pi} \int_0^{\pi/2} (\cos t\cot t)\, dt \, \sec t\,dt\, \Big(\csc t-\cot t\Big) \,dt\\
=&\,\frac{2}{\pi} \int_0^{\pi/2} (\csc t-\sin t)\, dt \, \sec t\,dt\, \Big(\csc t-\cot t\Big) \,dt\\
=&\,-X_0+\frac{2}{\pi} \int_0^{\pi/2} \csc t\, dt \, \sec t\,dt\, \Big(\csc t-\cot t\Big) \,dt\\
=&\,-X_0-(-1)^s \frac{2}{\pi} \int_0^1 \Big(2\tx_{-1}-\tx_{i}-\tx_{-i}\Big)\ta \td_{-1,1}, 
\end{align*}
by the change of variables $t\to \sin^{-1}(1-t^2)/(1+t^2)$ at the last step for both $X_0$ and $X_1$. Similarly,
\begin{align*}
Y_0:=&\, \frac{2}{\pi} \int_0^{\pi/2} \sin t\, dt \,dt\,d(\sec t) \Big(\csc t-\cot t\Big) \,dt\\
=&\, \frac{2}{\pi} \int_0^{\pi/2} \cos t\, dt\,d(\sec t) \Big(\csc t-\cot t\Big) \,dt\\
=&\,\frac{2}{\pi} \int_0^{\pi/2} dt \Big(\csc t-\cot t\Big)\,dt
-\frac{2}{\pi} \int_0^{\pi/2} \Big(1-\sin t\Big)\Big(\csc t\sec t -\csc t\Big)\,dt \\
=&\,\frac{2}{\pi} \int_0^1 i \Big(2\tx_{-1}-\tx_{i}-\tx_{-i}\Big) \td_{-i,i}+i \td_{-i,i}-2\tx_{-1}, \\
Y_1:=&\, \frac{2}{\pi} \int_0^{\pi/2} \sin t\, dt \cot t\,dt \,dt\,d(\sec t) \Big(\csc t-\cot t\Big) \,dt\\
=&\, \frac{2}{\pi} \int_0^{\pi/2} (\cos t\cot t)\, dt \,dt\,d(\sec t) \Big(\csc t-\cot t\Big) \,dt\\
=&\, -Y_0+\frac{2}{\pi} \int_0^{\pi/2} \csc t\, dt \,dt\,d(\sec t) \Big(\csc t-\cot t\Big) \,dt\\
=&\, -Y_0+\frac{2}{\pi} \int_0^{\pi/2} \csc t\, dt \Big[\sec t\, dt\, \Big(\csc t-\cot t\Big)\, dt -dt\, \Big(\csc t\sec t-\csc t\Big) \,dt\Big]\\
 = &\, -Y_0+\frac{2}{\pi} \int_0^1 \Big(2\tx_{-1}-\tx_{i}-\tx_{-i}\Big) \ta \td_{-1,1}
 -\frac{2}{\pi} \int_0^{\pi/2} i \Big(\ta+2\tx_{-1}\Big)\td_{-i,i} \td_{-1,1} .
\end{align*}
Thus
\begin{align*}
I_1=X_1-Y_1-I_0=&\,
Y_0-X_0-I_0+\frac{2}{\pi}\int_0^1 \Big(2\tx_{-1}-\tx_{i}-\tx_{-i}\Big)\ta \td_{-1,1}\\
&\,-\frac{2}{\pi} \int_0^1 \Big(2\tx_{-1}-\tx_{i}-\tx_{-i}\Big) \ta \td_{-1,1}
+\frac{2}{\pi} \int_0^{\pi/2} i \Big(\ta+2\tx_{-1}\Big)\td_{-i,i} \td_{-1,1} \\
=&\,
-I_0+\frac{2}{\pi} \int_0^1 i \Big(2\tx_{-1}-\tx_{i}-\tx_{-i}\Big) \td_{-i,i}+i \td_{-i,i}-2\tx_{-1} \\
&\, -\frac{2}{\pi} \int_0^1 i \Big(2\tx_{-1}-\tx_{i}-\tx_{-i}\Big) \td_{-i,i}
+\frac{2}{\pi} \int_0^{\pi/2} i \Big(\ta+2\tx_{-1}\Big)\td_{-i,i} \td_{-1,1}.
\end{align*}
 Note that $I_0$ is given by \cite[Example B.8]{XuZhao2022}. Hence we obtain the following formula by using Au's Mathematica package of colored multiple 
 zeta values \cite{Au2020}: 
\begin{align*}
 &\,\sum_{k>m>0} \left( \frac{1}{16} \right)^{k} \binom{2 k}{k}^2 \frac{1}{(2k-1)^3(2m)} \\
= &\,\frac{2}{\pi} \bigg(\frac5{32}\pi^3 - 4 \mathcal{G} -2G(1+2\ln2)+6\ln2+\frac{\pi}8\Big(8\ln2+\ln^22-12\Big) \bigg).
\end{align*}
Now, using partial fraction decomposition, we see that
\begin{align*}
&\,\sum_{k>0} \left( \frac{1}{16} \right)^{k} \binom{2 k}{k}^2 \frac{1}{(2k-1)^3 (2k)}\\
=&\,\sum_{k>0} \left( \frac{1}{16} \right)^{k} \binom{2 k}{k}^2 \left(\frac{1}{(2k-1)^3}-\frac{1}{(2k-1)^2}+\frac{1}{2k-1}-\frac{1}{2k}\right)\\
=&\,\frac{2}{\pi}\left(\frac{\pi}{2}+2G-3\right)
-\frac2{\pi}\Big(2-\frac{\pi}2\Big)
+\frac{2}{\pi} \Big(\frac{\pi}{2}-1\Big)
-\frac{2}{\pi} (\pi\ln2-2G) \\
=& \frac2{\pi}\Big( \frac{3\pi}2+4G-\pi\ln2-6 \Big) 
\end{align*}
by Examples B.1 and B.3 of \cite{XuZhao2022}. Therefore,
\begin{align*}
&\,\sum_{k>0} \left( \frac{1}{16} \right)^{k} \binom{2 k}{k}^2 \frac{H_k}{(2k-1)^3}\\
=&\,2\sum_{k>0} \left( \frac{1}{16} \right)^{k} \binom{2 k}{k}^2 \frac{1}{(2k-1)^3 (2k)}
+2\sum_{k>m>0} \left( \frac{1}{16} \right)^{k} \binom{2 k}{k}^2 \frac{1}{(2k-1)^3(2m)} \\
=&\, \frac{4}{\pi} \bigg(\frac5{32}\pi^3 - 4
 \mathcal{G} +2G(1-2\ln2)+6\ln2+\frac{\pi}8 \ln^22 -6 \bigg).
\end{align*}
\noindent So, we obtain that:
\begin{equation}\label{afterJay}
 \sum _{k=1}^{\infty} \left( \frac{1}{16} \right)^{k} \binom{2 k}{k}^2 \frac{H_k}{2k+1}
 = -\frac{16 G \ln (2)}{\pi }-\frac{16 \mathcal{G} }{\pi }+\frac{5 \pi ^2}{8}+\frac{\ln
 ^2(2)}{2}.
\end{equation}
 Thus, according to \eqref{Paulbrilliant}, we find that:
\begin{equation}\label{odd2kp1}
 \sum _{k=1}^{\infty} \left( \frac{1}{16} \right)^{k} \binom{2 k}{k}^2 \frac{O_{k} }{2 k+1}
 = -\frac{4 G \ln (2)}{\pi }-\frac{16 \mathcal{G} }{\pi }+\frac{3 \pi ^2}{8}+\frac{\ln
 ^2(2)}{2}.
\end{equation}
 So, according to our relation for \eqref{nocentral}, as derived from \cite{Campbell2022}, we obtain the desired result. 
\end{proof}

\section{The general case}\label{sectionnotConclusion}
 We will now derive an explicit evaluation for \eqref{mainproblem}, for an arbitrary parameter $\lambda \in \mathbb{N}$. We also take care of the open 
 problem from \cite{WangChu2022} of evaluating 
\begin{align}
 \sum_{n=1}^{\infty} \left( \frac{1}{16} \right)^{n} \binom{2n}{n}^2
 \frac{O_{n} }{(1 + 2n - 2 \lambda)^2}. \label{firstproblem} 
 \end{align}
 Chu \cite{Chu2023} had recently employed a hypergeometric approach to evaluate 
\begin{align}
 \sum_{n=1}^{\infty} \left( \frac{1}{16} \right)^{n} \binom{2n}{n}^2 
 \frac{O_{n}^2 }{1 + 2n - 2 \lambda}, \label{secondproblem}
 \end{align}
 and we will also need an evaluation for \eqref{secondproblem}, in our solution to the main problem of evaluating \eqref{mainproblem}. We will also 
 need an evaluation for 
\begin{equation}\label{easycase}
 \sum_{n=1}^{\infty} \left( \frac{1}{16} \right)^{n} \binom{2n}{n}^2
 \frac{O_{n} }{1 + 2n - 2 \lambda}, 
 \end{equation}
 noting that series of the form indicated in \eqref{easycase} were also evaluated by 
 Wang and Chu in \cite{WangChu2022}, again via coefficient extractions. 

 Defining 
\[ A(m) := \begin{cases} 
 \displaystyle -\frac{16^{m}}{ 2\pi m^2 \binom{2m}{m}^2 } \sum_{k=0}^{m - 
 	1} \left( \frac{1}{16} \right)^{k} \binom{2k}{k}^2 & \text{if $m \in \mathbb{N}$} \\ 
 \displaystyle \frac{4 G}{\pi } & \text{if $m = 0$} 
 \end{cases}
\]
 and 
 \[ B(m) := \begin{cases} 
 \displaystyle  \frac{16^{m}}{m^2 \binom{2m}{m}^2 } \sum_{k = 
 0}^{m - 1} \left( \frac{1}{16} \right)^{k} \binom{2k}{k}^2 \frac{k A(k) +\tfrac{1}{\pi}}{2k+1} 
 & \text{if $m \in \mathbb{N}$} \\ 
 \displaystyle -\frac{16 \mathcal{G}}{\pi} + \frac{3\pi^2}{8} + \frac{\ln^2(2)}{2} & \text{if $m = 0$,}
 \end{cases}
\]
 this leads us toward the following Lemma. The $m = 0$ case for \eqref{Lemma2} has been 
 established in \cite{CampbellLevrieNimbran2021,CantariniDAurizio2019}. 

\begin{lemma}\label{Lemma}
For $m \in \mathbb{N}_{0}$, we have that
\begin{align}
 \sum_{n = 0}^{\infty} \left( \frac{1}{16} \right)^{n} \binom{2n}{n}^2 \frac{1}{2n+1-2m} & = 
 A(m), \label{Lemma1}\\ 
 \sum_{n=0}^{\infty} \left( \frac{1}{16} \right)^{n} \binom{2n}{n}^2 \frac{1}{(2n+1-2m)^2} & = 
 B(m). \label{Lemma2}
\end{align}
\end{lemma}

     The proof of these two results is similar to that of \eqref{SfunctionRamanujan}. See also \cite{Nimbran2015}. With regard to our applying the  
 $A$- and $B$-sequences above toward our fully solving open problems from \cite{WangChu2022}, our results as in Theorem \ref{fullsolution} are 
 also of interest due to how the \emph{Landau constants} $\sum_{i=0}^{n} \left( \frac{1}{16} \right)^{i} \binom{2i}{i}^2$ are heavily involved in our 
 work; see \cite{CampbellChen2022} and references therein for related uses of the Landau constants. We see that the $A$-expression defined in 
 \eqref{Lemma1} is, up to a scalar multiple by a closed-form expression, equal to a Landau constant for all $m$, and similarly with the Ramanujan 
 summation in \eqref{Sfunction}. The history of Ramanujan's $S$-function \cite{Adamchik2002}, the importance of the Landau constants in complex 
 analysis, and the ongoing interest in computational problems concerning approximating the Landau constants all contribute to the interest in 
 how our explicit, finite sum evaluations as in Theorem \ref{fullsolution} may be formulated in a natural way in terms of the above indicated $A$-sequence. 

\begin{theorem} \label{theorem2}
For all $\lambda \in \mathbb{N}$ we have
\begin{equation}
\sum_{n=1}^{\infty} \left( \frac{1}{16} \right)^{n} \binom{2n}{n}^2
\frac{O_{n} }{1 + 2n - 2 \lambda} = -\ln(2) A(\lambda)+ \frac{16^{\lambda}}{ 2\lambda^2 \binom{2\lambda}{\lambda}^2 } \sum_{k=0}^{\lambda - 1} \left( \frac{1}{16} \right)^{k} \binom{2k}{k}^2 k A(k). \label{th1}
\end{equation}
\end{theorem}
	
\begin{proof}
 To determine a recursion for \eqref{easycase}, we start from
 \begin{align}\label{recur6}
 \sum_{n=1}^{\infty} & \left( \frac{1}{16} \right)^{n} \binom{2n}{n}^2
 \frac{O_{n} }{1 + 2n - 2 \lambda } =
 \sum_{n=0}^{\infty} \left( \frac{1}{16} \right)^{n} \binom{2n}{n}^2
 \frac{O_{n+1}-\frac{1}{2n+1}}{1 + 2n - 2 \lambda} \\ & =
 \sum_{n=0}^{\infty} \left( \frac{1}{16} \right)^{n} \binom{2n}{n}^2
 \frac{O_{n+1}}{1 + 2n - 2 \lambda}\nonumber -\sum_{n=0}^{\infty} \left( \frac{1}{16} \right)^{n} \binom{2n}{n}^2 \frac{1}{(2n+1)(1 + 2n - 2 \lambda)} .
 \end{align}
 For the second term at the right-hand side  we use \eqref{Lemma1} and the well known series
 \[
 \sum_{n=0}^{\infty} \left( \frac{1}{16} \right)^{n} \binom{2n}{n}^2
 \frac{1}{2n+1} = \frac{4G}{\pi} .
 \]
 Using partial fractions, we then obtain:
 \[
 \sum_{n=0}^{\infty} \left( \frac{1}{16} \right)^{n} \binom{2n}{n}^2
 \frac{1}{(2n+1)(1 + 2n - 2 \lambda)} = \frac{1}{2\lambda} A(\lambda) - \frac{1}{2\lambda} \frac{4G}{\pi} .
 \]
For the first term at the right-hand side  we use reindexing:
 $$ \sum_{n=0}^{\infty} \left( \frac{1}{16} \right)^{n} \binom{2n}{n}^2
 \frac{O_{n+1}}{1 + 2n - 2 \lambda} = 4\sum_{n=0}^{\infty} \left( \frac{1}{16} \right)^{n} \binom{2n}{n}^2 O_{n}
 \frac{n^2}{(2n-1)^2(1 + 2n - 2 \lambda-2)} . $$
The last factor has the following partial fraction expansion:
 $$\frac{n^2}{(2n-1)^2 (1 + 2n - 2 \lambda-2)} = \frac{1}{16\lambda^2} \frac{(1+2\lambda)^2}{2n+1-2\lambda-2} - \frac{1}{16 \lambda^2}\frac{1+4\lambda}{2n-1} -\frac{1}{8\lambda}\frac{1}{(2n-1)^2} $$
 leading to the following two known series (see \cite{WangChu2022}) 
\begin{align*}
& - \frac{1+4\lambda}{16 \lambda^2} \sum_{n=0}^{\infty} \left( \frac{1}{16} \right)^{n} \binom{2n}{n}^2 \frac{O_n}{2n-1} = - \frac{1+4\lambda}{16 \lambda^2}\frac{2\ln(2)}{\pi}, \hspace{1cm} \\
& -\frac{1}{8\lambda}\sum_{n=0}^{\infty} \left( \frac{1}{16} \right)^{n} \binom{2n}{n}^2 \frac{O_n}{(2n-1)^2} = -\frac{1}{2\lambda}\frac{G 
 - \ln(2)}{\pi}, \hspace{1cm} 
\end{align*}
 and the series $$ \frac{(1+2\lambda)^2}{16\lambda^2} \sum_{n=1}^{\infty} \left( \frac{1}{16} \right)^{n} \binom{2n}{n}^2 \frac{O_{n} }{1 + 2 n 
 - 2 \lambda-2 } . $$ Setting $$ C(m):=\sum_{n=1}^{\infty} \left( \frac{1}{16} \right)^{n} \binom{2n}{n}^2 \frac{O_{n} }{1 + 2n - 2 m } $$ 
 \eqref{recur6} can be rewritten as: $$ C(\lambda+1)=\frac{4\lambda^2}{(1+2\lambda)^2} C(\lambda) + \frac{2}{(1+2\lambda)^2} 
 \left(\lambda A(\lambda) + \frac{\ln(2)}{\pi} \right) \ \ \mbox{with} \ \ C(1)=\frac{2\ln(2)}{\pi}. $$ To solve this recurrence, note that $$z(\lambda) := 
 2\left( \frac{1}{16} \right)^{\lambda} \binom{2\lambda}{\lambda}^2 \lambda^2 C(\lambda)$$ is a solution of the recurrence $$ z(\lambda+1) = 
 z(\lambda) + \left( \frac{1}{16} \right)^{\lambda} \binom{2\lambda}{\lambda}^2 \left(\lambda A(\lambda) + \frac{\ln(2)}{\pi} \right) 
 \ \ \ \text{with} \ \ \ z(0)=0. $$ An explicit evaluation for this solution is given by $$z(\lambda) = \sum_{k = 0}^{\lambda - 
 1} \left( \frac{1}{16} \right)^{k} \binom{2k}{k}^2 \left( k A(k) + \frac{\ln(2)}{\pi} \right), $$ leading to: 
\begin{equation}\label{C(m)}
 C(\lambda) =\frac{16^{\lambda}}{ 2\lambda^2 \binom{2\lambda}{\lambda}^2 } \sum_{k = 
 0}^{\lambda - 1} \left( \frac{1}{16} \right)^{k} \binom{2 k}{k}^2 
 \left( k A( k ) + \frac{\ln(2)}{\pi} \right) 
\end{equation}
which is equivalent with \eqref{th1}.
\end{proof}

\begin{theorem}
 For all $\lambda \in \mathbb{N}$ we have 
\begin{equation}
	\sum_{n=1}^{\infty} \left( \frac{1}{16} \right)^{n} \binom{2n}{n}^2
	\frac{O_{n} }{(1 + 2n - 2 \lambda)^2} = \frac{16^{\lambda}}{\lambda^2 \binom{2\lambda}{\lambda}^2 } \left( \frac{G-\ln(2)}{\pi}+\sum_{k=1}^{\lambda - 1} \left( \frac{1}{16} \right)^{k} \binom{2k}{k}^2 D(k) \right) \label{th2}
\end{equation}
where $C(m)$ is given by \eqref{C(m)} and
	$$D(m):= \frac{1}{2m+1} \left( m C(m) - \frac{2m-1}{4} A(m) - \frac{\ln (2)}{\pi} \right) + \frac{m}{2} B(m).$$ 
\end{theorem}

\begin{proof} In this case we start from:
	 \begin{align}\label{recur4}
	\sum_{n=1}^{\infty} & \left( \frac{1}{16} \right)^{n} \binom{2n}{n}^2
	\frac{O_{n} }{(1 + 2n - 2 \lambda)^2 } =
	\sum_{n=0}^{\infty} \left( \frac{1}{16} \right)^{n} \binom{2n}{n}^2
	\frac{O_{n+1}-\frac{1}{2n+1}}{(1 + 2n - 2 \lambda)^2} \\ & =
	\sum_{n=0}^{\infty} \left( \frac{1}{16} \right)^{n} \binom{2n}{n}^2
	\frac{O_{n+1}}{(1 + 2n - 2 \lambda)^2}\nonumber -\sum_{n=0}^{\infty} \left( \frac{1}{16} \right)^{n} \binom{2n}{n}^2 \frac{1}{(2n+1)(1 + 2n - 2 \lambda)^2} .
	\end{align}
The rest of the proof is very similar to the proof of the previous theorem, but in this case we need the following two results:
\[
\sum_{n=0}^{\infty} \left( \frac{1}{16} \right)^{n} \binom{2n}{n}^2 \frac{1}{(1 + 2n - 2 \lambda)^2} = B(\lambda)
\]
and
\[
\sum_{n=0}^{\infty} \left( \frac{1}{16} \right)^{n} \binom{2n}{n}^2 \frac{O_n}{1 + 2n - 2 \lambda} = C(\lambda).
\]
Let us define
\[
E(m):=\sum_{n=1}^{\infty} \left( \frac{1}{16} \right)^{n} \binom{2n}{n}^2
\frac{O_{n} }{(1 + 2n - 2 m)^2}
\]
then the corresponding recurrence relation can be written as:
\[
(\lambda+1)^2 \left( \frac{1}{16} \right)^{\lambda+1} \binom{2\lambda+2}{\lambda+1}^2 E(\lambda+1)
= \lambda^2 \left( \frac{1}{16} \right)^{\lambda} \binom{2\lambda}{\lambda}^2 E(\lambda) + \left( \frac{1}{16} \right)^{\lambda} \binom{2\lambda}{\lambda}^2 D(\lambda)
\]
with $D(m)$ as in the statement of the theorem and with $E(1) = \frac{4G - 4\ln(2)}{\pi}$.
\end{proof}

\begin{theorem} 
For all $\lambda \in \mathbb{N}$ we have
\begin{equation}
	\sum_{n=1}^{\infty} \left( \frac{1}{16} \right)^{n} \binom{2n}{n}^2
	\frac{O_{n}^2 }{1 + 2n - 2 \lambda} = \frac{16^{\lambda}}{\lambda^2 \binom{2\lambda}{\lambda}^2 } \left( \frac{48G-\pi^2 -24\ln^2(2)}{48\pi}+\sum_{k=1}^{\lambda - 1} \left( \frac{1}{16} \right)^{k} \binom{2k}{k}^2 F(k) \right) \label{th4}
\end{equation}
where $C(m)$ is given by \eqref{C(m)} and
	$$F(m):= m C(m) + \frac{1}{4} A(m) - \frac{\pi^2+24\ln^2 (2)}{48\pi} .$$ 
\end{theorem}

\begin{proof}
 We start with $$\sum_{n=1}^{\infty} \left( \frac{1}{16} \right)^{n} \binom{2n}{n}^2 \frac{O_{n}^2 }{1 + 2n - 2 \lambda } = \sum_{n = 
 0}^{\infty} \left( \frac{1}{16} \right)^{n} \binom{2n}{n}^2 \frac{\left(O_{n+1}-\frac{1}{2n+1}\right)^2}{1 + 2n - 2 \lambda} .$$ 
 We proceed in the same way as before. After expanding using partial fractions, we need the sums of some additional series: 
 \[ \sum_{n=0}^{\infty} \left( \frac{1}{16} \right)^{n} \binom{2n}{n}^2 
 \frac{1}{(2n+1)^2} = - \frac{16 \mathcal{G} }{\pi } + \frac{3 \pi ^2}{8} +\frac{\ln ^2(2)}{2} \quad \text{(see \eqref{eq:5})}, 
 \] 
 \[ 
 \sum_{n=0}^{\infty} \left( \frac{1}{16} \right)^{n} \binom{2n}{n}^2 
 \frac{O_n^2}{(2n-1)^2} = 
 -\frac{8 G}{\pi }-\frac{8 G \ln (2)}{\pi }-\frac{48 \mathcal{G}}{\pi }+\frac{9 \pi^2}{8}+\frac{\pi }{6}+\frac{4 \ln ^2(2)}{\pi }+\frac{3 \ln ^2(2)}{2} 
 \]
(see Theorem \ref{maintheorem}), and
\[
\sum_{n=0}^{\infty} \left( \frac{1}{16} \right)^{n} \binom{2n}{n}^2
\frac{O_n}{(2n-1)^3} = -\frac{4 G (\ln (2)+2)}{\pi }-\frac{32 \mathcal{G} }{\pi }+ \frac{6 \ln (2)}{\pi}+\frac{3 \pi ^2}{4}+\ln^2(2) .
\]
The sum of this last series can be found by reindexing:
\begin{align*}
\sum_{n=1}^{\infty} & \left( \frac{1}{16} \right)^{n} \binom{2n}{n}^2
\frac{O_n}{(2n-1)^3} =
\sum_{n=1}^{\infty} \left( \frac{1}{16} \right)^{n} \frac{4(2n-1)^2}{n^2}\binom{2n-2}{n-1}^2
\frac{O_{n-1}+\tfrac{1}{2n-1}}{(2n-1)^3} \\
& =
\frac{1}{4}\sum_{n=1}^{\infty} \left( \frac{1}{16} \right)^{n} \binom{2n}{n}^2
\frac{O_n}{(n+1)^2 (2n+1)} + \frac{1}{4}\sum_{n=1}^{\infty} \left( \frac{1}{16} \right)^{n} \binom{2n}{n}^2
\frac{1}{(n+1)^2 (2n+1)^2}.
\end{align*}
 In addition to the series used before, here we also need these: 
\begin{align*}
 & \sum_{n=1}^{\infty} \left( \frac{1}{16} \right)^{n} \binom{2n}{n}^2 
 \frac{1}{(n+1)^i} \ (i=1,2) \quad \text{(see \cite{Nimbran2015})}, \\ 
 & \sum_{n=1}^{\infty} \left( \frac{1}{16} \right)^{n} \binom{2n}{n}^2 
 \frac{O_n}{(n+1)^i} \ (i=1,2) \quad \text{(see \cite{WangChu2022})}, 
\end{align*}
 and $$ \sum_{n=1}^{\infty} \left( \frac{1}{16} \right)^{n} \binom{2n}{n}^2 
 \frac{O_n}{2n+1}, $$ 
 which is given in \eqref{odd2kp1}. Defining 
 \[ G(m):=\sum_{n=1}^{\infty} \left( \frac{1}{16} \right)^{n} \binom{2n}{n}^2 
 \frac{O_{n}^2 }{1 + 2n - 2 m } \] 
 we find the same type of recurrence for $G$.    
\end{proof}

 The following is a main result in our work, as it provides a full solution to an especially difficult open problem from \cite{WangChu2022}, recalling the 
 nontrivial nature of the base case covered in Section \ref{sectionbase}. 

\begin{theorem}\label{fullsolution}
 For $\lambda \in \mathbb{N}$,
\begin{equation}
	\sum_{n=1}^{\infty} \left( \frac{1}{16} \right)^{n} \binom{2n}{n}^2
	\frac{O_{n}^2 }{(1 + 2n - 2 \lambda)^2} = \frac{16^{\lambda}}{\lambda^2 \binom{2\lambda}{\lambda}^2 } \left( K + \sum_{k=1}^{\lambda - 1} \left( \frac{1}{16} \right)^{k} \binom{2k}{k}^2 H(k) \right) \label{th5}
	\end{equation}
 with
	$$ K:=\frac{1}{4} \left( -\frac{8 G}{\pi }-\frac{8 G \ln (2)}{\pi }-\frac{48 \mathcal{G}}{\pi } + 
 \frac{9 \pi^2}{8}+\frac{\pi }{6}+\frac{4 \ln ^2(2)}{\pi }+\frac{3 \ln ^2(2)}{2} \right)$$
	and
	$$H(m):= \frac{ 2m (G(m)-C(m)) + C(m) - A(m)}{2(2m+1)}+ \frac{\pi^2+24\ln^2 (2)}{24(2m+1)\pi} + \frac{1}{4} B(m) + m E(m).$$
\end{theorem}

\begin{proof} The method is the same as in the proof of the previous theorems. We start with: $$\sum_{n=1}^{\infty} \left( \frac{1}{16} \right)^{n} 
 \binom{2n}{n}^2 \frac{O_{n}^2 }{(1 + 2n - 2 \lambda)^2 } = 
 \sum_{n = 0}^{\infty} \left( \frac{1}{16} \right)^{n} \binom{2n}{n}^2 
 \frac{\left(O_{n+1}-\frac{1}{2n+1}\right)^2}{(1 + 2n - 2 \lambda)^2 } .$$ 
 Expanding the square in the numerator and reindexing the series containing $O_n$ leads to a sum of three series: 
 \begin{align*}
 4\sum_{n=0}^{\infty} & \left( \frac{1}{16} \right)^{n} \binom{2n}{n}^2 O_{n}^2
 \frac{n^2}{(2n-1)^2(1 + 2n - 2 \lambda-2)^2}, \\ 
 -8\sum_{n=0}^{\infty} & \left( \frac{1}{16} \right)^{n} \binom{2n}{n}^2 O_{n} 
 \frac{n^2}{(2n-1)^3(1 + 2n - 2 \lambda-2)^2}, \\ 
 \sum_{n=0}^{\infty} &\left( \frac{1}{16} \right)^{n} \binom{2n}{n}^2 
 \frac{1}{(2n+1)^2(1 + 2n - 2 \lambda-2)^2}. 
 \end{align*}
 We now use partial fractions as before. One of the series we obtain is the series \eqref{th5} with $\lambda$ replaced by $\lambda+1$. All the other ones we have already encountered in the previous theorems.
 Setting
 \[
 J(m):=\sum_{n=1}^{\infty} \left( \frac{1}{16} \right)^{n} \binom{2n}{n}^2
 \frac{O_{n}^2 }{(1 + 2n - 2 m)^2 },
 \]
 we again find that $J$ satisfies a first order recurrence:
 \[
 (\lambda+1)^2 \left( \frac{1}{16} \right)^{\lambda+1} \binom{2\lambda+2}{\lambda+1}^2 J(\lambda+1)
 = \lambda^2 \left( \frac{1}{16} \right)^{\lambda} \binom{2\lambda}{\lambda}^2 J(\lambda) + \left( \frac{1}{16} \right)^{\lambda} \binom{2\lambda}{\lambda}^2 H(\lambda)
 \]
 with $J(1)=4K$ by Theorem \ref{maintheorem}. Solving this recurrence in the same way as in Theorem \ref{theorem2} leads to the desired result.
\end{proof}

\begin{example}
 In addition to the base case covered in Section \ref{sectionbase}, we obtain,
 for the $\lambda = 2$, $\lambda = 3$, and $\lambda = 4$ cases,
 the following:
 \begin{align*}
 & \sum_{k=1}^{\infty} \left( \frac{1}{16} \right)^{k} \binom{2k}{k}^2 \frac{ O_{k}^{2} }{ (2k-3)^2 }
 = -\frac{32 G}{27 \pi} - \frac{32 G \ln 2}{9 \pi} \\
 & \hspace{5cm} - \frac{64 \mathcal{G}}{3 \pi}
 + \frac{\pi^2}{2} + \frac{11 \pi}{162} + \frac{16}{27 \pi} + \frac{44 \ln^2 2}{27 \pi}
 + \frac{2 \ln^2 2}{3} - \frac{52 \ln 2}{27 \pi}, \\
 & \sum_{n=1}^{\infty} \left( \frac{1}{16} \right)^{n} \binom{2n}{n}^2 \frac{O_{n}^2}{(2n-5)^2} =
 \frac{32 \ln^2(2)}{75}+\frac{8 \pi^{2}}{25}+\frac{3388 \ln^2(2)}{3375 \pi}-\frac{512 G \ln \! \left(2\right)}{225 \pi} \\
 & \hspace{5cm} +\frac{847 \pi}{20250}-\frac{5908 \ln \! \left(2\right)}{3375 \pi}-\frac{256 G}{3375 \pi}-\frac{1024 \mathcal{G}}{75 \pi}+\frac{2624}{3375 \pi}, \\
 & \sum_{n=1}^{\infty} \left( \frac{1}{16} \right)^{n} \binom{2n}{n}^2 \frac{O_{n}^2}{(2n-7)^2} =
 \frac{384 \ln^2(2)}{1225}+\frac{288 \pi^{2}}{1225}+\frac{92804 \ln^2(2)}{128625 \pi}-\frac{2048 G \ln \! \left(2\right)}{1225 \pi}\\
 & \hspace{3cm} +\frac{23201 \pi}{771750}-\frac{65284 \ln \! \left(2\right)}{42875 \pi}+\frac{11264 G}{42875 \pi}-\frac{12288 \mathcal{G}
 }{1225 \pi}+\frac{188144}{231525 \pi}.
 \end{align*}
\end{example}

\section{Further evaluations and concluding remarks}\label{sectionFurther}
 Our methods have involved the evaluation of series involving summand factors of the form $$ \left( \frac{1}{16} \right)^{k} \binom{2 k}{k}^{2} 
 \frac{1}{2k+1} $$ for $k \in \mathbb{N}_{0}$, so it seems worthwhile to consider further evaluations for such sums and 
 related summations. In this regard, setting $\beta(4)=\sum_{n=0}^\infty \frac{(-1)^n}{(2n+1)^4}$, 
 let us reproduce the following
 identities from \cite{CampbellLevrieNimbran2021,CantariniDAurizio2019}:
\begin{align}
& \sum _{k=0}^{\infty } \left( \frac{1}{16} \right)^{k} \binom{2k}{k}^2
\frac{1}{2 k+1} = \frac{4G}{\pi}, \label{eq:4}\\
& \sum _{k=0}^{\infty } \left( \frac{1}{16} \right)^{k} \binom{2k}{k}^2
\frac{1}{(2 k+1)^2} = - \frac{16 \mathcal{G} }{\pi } + \frac{3 \pi ^2}{8} +\frac{\ln ^2(2)}{2}, \label{eq:5}\\
& \sum _{k=0}^{\infty } \left( \frac{1}{16} \right)^{k} \binom{2k}{k}^2
\frac{1}{(2 k+1)^3} =
-64\frac{\Im \left(\text{Li}_{4} (1-i)\right)}{\pi} - \pi^2 \ln(2)-\frac{2}{3} \ln^3 (2) - 48 \frac{\beta(4)}{\pi}. \label{Li4}
\end{align}
   Such evaluations are closely related to the evaluation    of series of the form $\sum_{k=1}^\infty \frac{\binom{2k}{k}^2}{k^n \, 16^k}$.    For example,   
   consider the following analogue of \eqref{eq:5}.  

\begin{theorem}\label{thm:alpha2}
 The evaluation $$ {\sum_{k=1}^\infty \frac{\binom{2k}{k}^2}{k^2 \, 16^k} = \frac{3 \pi^2}{2} - \frac{ 64 \mathcal{G}}{\pi} -
 6 \ln^2 (2)} $$ holds true.
\end{theorem}

\begin{proof}
 For our proof, we have to calculate the following: 
\begin{equation}\label{break2int}
 \sum_{k=1}^\infty \frac{\binom{2k}{k}^2}{k^2 \, 16^k} \cdot \frac{\pi}{2}= \int_0^{\pi/2}
2\, {\rm Li}_2 (\tfrac{1-\cos (t)}{2}) {dt}- 4\int_0^{\pi/2}\ln^2 (\cos(\tfrac{t}{2}) ) {dt}. 
\end{equation}
 The second integral in \eqref{break2int} is equivalent with
\begin{equation}
 8 \int_0^{\pi/4}\ln^2 (\cos(t)) {dt}= -\frac{7\pi^3}{48}+\frac{7}{4}\pi \ln^2(2)+8\mathcal{G}- 4 G \ln(2), \label{eq:2.3}
\end{equation}
which can be found in \cite{CampbellLevrieNimbran2021}. For the first one, we use partial integration:
\begin{align*}
\int_0^{\pi/2} 2\, {\rm Li}_2 (\tfrac{1-\cos (t)}{2}) {dt} & = \frac{1}{12} \pi^3
- \frac{1}{2}\pi \ln^2(2) + 2\int_0^{\pi/2} t \cot (\tfrac{t}{2}) \ln (\cos^2 (\tfrac{t}{2})) {dt} \\
& = \frac{1}{12} \pi^3
- \frac{1}{2}\pi \ln^2(2) + 8\int_0^{\pi/4} t \cot (t) \ln (\cos^2 (t)) {dt} .
\end{align*}
For this integral we use the substitution $u = \tan (t)$:
\begin{align*}
\int_0^{\pi/4} t \cot (t) \ln (\cos^2 (t)) {dt} & = -\int_0^1 \arctan (u) \ln (1+u^2) \frac{1}{u (1+u^2)} {du} \\
 & = -\int_0^1 \arctan (u) \ln (1+u^2) \left( \frac{1}{u}-\frac{u}{1+u^2} \right) {du}.
\end{align*}
 The first integral can be found in \cite{CampbellLevrieNimbran2021}. The second one requires one 
 partial integration in combination with \cite{CampbellLevrieNimbran2021}: 
\begin{align*}
& \int_{0}^{\pi/4} \ln^2 (\sin(t)) dt = - \mathcal{G} +\frac{1}{2} G \ln (2) +\frac{23}{384} \pi ^3 +\frac{9}{32} \pi \ln^2(2)
\end{align*}
 and \eqref{eq:2}. Bringing everything together, we obtain the desired result.
\end{proof}

\begin{remark}
Another method to evaluate the series in Theorem~\ref{thm:alpha2} is given in \cite[Example B.1]{XuZhao2022} using different basis elements.
 Moreover, Xu and Zhao proved in loc.\ cit.\ that
\begin{equation}\label{equ:re2}
 \sum_{k=1}^\infty \frac{\binom{2k}{k}^2}{k^3 \, 16^k}
=- \frac{512}{\pi} \Im \left(\text{Li}_4\Big(\frac{1+i}{2}\Big) \right)
- 6 \pi^2 \ln(2) +8 \ln^3 (2) + 384 \frac{\beta(4)}{\pi} + 4\zeta(3).
\end{equation}
\end{remark}

 Noting that $$\sum_{k=1}^\infty \frac{\binom{2k}{k}^2}{k^n \, 16^k} = \sum_{k=0}^\infty \frac{4(2k+1)^2\binom{2k}{k}^2}{(k+1)^{n+2} \, 
 16^{k+1}} = \sum_{k=0}^\infty \frac{4(2(k+1)-1)^2\binom{2k}{k}^2}{(k+1)^{n+2} \, 16^{k+1}}, $$ we obtain, using previous results, that 
\begin{align*}
 \sum_{k=0}^\infty \frac{\binom{2k}{k}^2}{(k+1)^3 \, 16^k} & = \frac{48}{\pi} - 16 - \frac{32G}{\pi} - 16\ln (2), \\
 \sum_{k = 0}^\infty
 \frac{\binom{2 k}{k}^2}{(k+1)^4 \, 16^k} & = \frac{128}{\pi}-48-\frac{128G}{\pi}+64\ln(2) + 6 \pi^2 - \frac{256\mathcal{G}}{\pi}-24
 \ln^2 (2), \\
 \sum_{k=0}^\infty \frac{\binom{2k}{k}^2}{(k+1)^5 \, 16^k} & =
 \frac{320}{\pi}-128-\frac{384G}{\pi}+192\ln(2) + 24 \pi^2 + \frac{1024\mathcal{G}}{\pi} -96 \ln^2 (2) \\
 & - \frac{2048}{\pi} \Im \left(\text{Li}_4\Big(\frac{1+i}{2}\Big) \right) -24 \pi^2 \ln(2) + 32 \ln^3 (2) + \frac{1536\beta(4)}{\pi} + 16\zeta(3).
\end{align*}

 Our applications, as above, of series as in \eqref{eq:4}--\eqref{Li4} motivate the study and application of identities as below. 

\begin{theorem}\label{thm-Varz}
For any integer $p\geq 0$, we have
\begin{align}\label{equ-int-log-sin-1}
\sum_{n=0}^\infty \binom{2n}{n}\frac{z^{2n+1}}{(2n+1)^{p+1}}=\frac{\theta}{2} \frac{\ln^p(2\sin\theta)}{p!}+\frac{1}{4p!} \sum_{j=1}^p (-1)^{j-1} \binom{p}{j} \ln^{p-j}(2\sin\theta)\Ls_{j+1}(2\theta),
\end{align}
 where $\theta:=\arcsin(2z)$ and 
\begin{align}
 \Ls_j(\theta) := -\int_0^\theta \ln^{j - 1}\left(2\sin\frac{t}{2}\right) \, dt. 
\end{align}
\end{theorem}
\begin{proof}
 We begin with the elementary Maclaurin series expansion 
\begin{align}\label{binom-arcsin-1}
 \sum_{n=0}^\infty \binom{2n}{n} \frac{z^{2n+1}}{2n+1}=\frac1{2}\arcsin(2z), 
\end{align}
 which is the $p = 0$ case. Now, from \eqref{binom-arcsin-1}, one deduces for $p > 0$ that 
\begin{align*}
&\sum_{n=0}^\infty \binom{2n}{n}\frac{z^{2n+1}}{(2n+1)^{p+1}}=\frac{1}{2(p-1)!} \int_0^z \frac{\ln^{p-1}\left(\frac{z}{t}\right)\arcsin(2t)}{t}dt\\
&=\frac{1}{2(p-1)!} \int_0^z \frac{\ln^{p-1}\left(\frac{4z}{4t}\right)\arcsin(2t)}{4t}d(4t)\\
&=\frac{1}{2(p-1)!} \sum_{j=0}^{p-1} \binom{p-1}{j}(-1)^j \ln^{p-1-j}(4z) \int_0^z \frac{\ln^j(4t)\arcsin(2t)}{4t}d(4t)\\
&=\frac{1}{2(p-1)!} \sum_{j=0}^{p-1} \binom{p-1}{j}\frac{(-1)^j}{j+1} \ln^{p-1-j}(4z) \int_0^z \arcsin(2t)d\ln^{j+1}(4t)\\
&=\frac{1}{2(p-1)!} \sum_{j=0}^{p-1} \binom{p-1}{j}\frac{(-1)^j}{j+1} \ln^{p-1-j}(4z) \int_0^\theta xd\ln^{j+1}(2\sin x)\\
&=\frac{1}{2(p-1)!} \sum_{j=0}^{p-1} \binom{p-1}{j}\frac{(-1)^j}{j+1} \ln^{p-1-j}(4z) \left(\theta\ln^{j+1}(2\sin\theta)-\int_0^\theta \ln^{j+1}(2\sin x)dx\right)\\
&=\frac{1}{2(p-1)!}\theta \ln^p(2\sin\theta)\sum_{j=0}^{p-1} \binom{p-1}{j}\frac{(-1)^j}{j+1}+\frac{1}{4p!} \sum_{j=1}^p (-1)^{j-1} \binom{p}{j} \ln^{p-j}(2\sin\theta)\Ls_{j+1}(2\theta).
\end{align*}
Then, noting the fact that
\[\sum_{j=0}^{p-1} \binom{p-1}{j}\frac{(-1)^j}{j+1}=\frac1{p},\]
we obtain the desired evaluation.
\end{proof}

Setting $p\leq 3$ in \eqref{equ-int-log-sin-1}, we deduce {\color{black}(with $\theta:=\arcsin(2z)$)}
\begin{align}\label{equ-paulExample}
&\sum_{n=0}^\infty \binom{2n}{n} \frac{z^{2n+1}}{2n+1}=\frac1{2}\theta,\\
&\sum_{n=0}^\infty \binom{2n}{n} \frac{z^{2n+1}}{(2n+1)^2}=\frac1{2}\theta\ln(2\sin\theta)+\frac1{4}\cl_2(2\theta),\\
&\sum_{n=0}^\infty \binom{2n}{n} \frac{z^{2n+1}}{(2n+1)^3}=\frac{1}{4}\theta\ln^2(2\sin\theta)+\frac1{4}\ln(2\sin\theta)\cl_2(2\theta)-\frac{1}{8}\Ls_3(2\theta),
\end{align}
where we used the relation $\Ls_2(\theta)=\cl_2(\theta)$ and for any positive integer $m$ the Clausen function $\cl_m(\theta)$ is defined 
 as follows, for all $\theta \in [0, \pi]$: 
\begin{align*}
 &{\cl}_{2m-1}(\theta):=\sum_{n=1}^\infty \frac{\cos(n\theta)}{n^{2m-1}}\quad \text{and}\quad {\cl}_{2m}(\theta) := 
 \sum_{n = 1}^\infty \frac{\sin(n\theta)}{n^{2m}}. 
\end{align*}

If we replace $z$ by $\frac{\sin \theta}{2}$ in \eqref{equ-paulExample}, divide by $\sin \theta$, multiply by $\ln (\sin \theta)$ and integrate between 0 and $\frac{\pi}{2}$, we get the following result:
\begin{align*}
\int_0^{\pi/2} \frac{\theta}{\sin \theta} \ln (\sin \theta)\, {d\theta} & = \frac{\pi}{2}\sum_{k=0}^\infty \left( \frac{1}{16} \right)^{k} \binom{2k}{k}^2 (O_k - \frac{1}{2}H_k - \ln 2) \frac{1}{2k+1} \\
& = - 4 \mathcal{G} + \frac{\pi^3}{32} + \frac{\pi \ln^2 (2)}{8}
\end{align*}
using \eqref{Paulbrilliant}. Note that using integration techniques from \cite{Campbell2018} in combination with \eqref{afterJay}, we
 can prove the related result: $$\int_0^{\pi/2} \frac{\theta}{\sin \theta} \ln (\cos \theta) d \theta 
 = 2 G \ln (2)+4 \mathcal{G} -\frac{5 \pi ^3}{32}-\frac{ \pi \ln ^2(2)}{8}. $$

\begin{remark}
 The series in \eqref{mainproblem} may be reduced to a $\mathbb{Q}$-combination of the series treated in \cite{XuZhao2022}. We offer a sketch of a proof,
 as below, based on this alternate approach.
\end{remark}
First, for all $k \in \mathbb{N}$ and $\bfs=(s_1,\dots,s_d)\in \mathbb{N}^d$ we define the multiple $t$-sums as
\begin{align*}
t_k(\bfs):= \sum_{k\ge k_1>\dots>k_d>0} \frac{1}{(2k_1-1)^{s_1}\cdots (2k_d-1)^{s_d}}.
\end{align*}
They satisfy the stuffle relations. For example,
\begin{align*}
 t_k(1)^2= 2t_k(1,1)+ t_k(2).
\end{align*}
 Since $O_k=t_k(1)$ the $\gl=1$ case of (2) is
\begin{align*}
\sum_{k>0} \left( \frac{1}{16} \right)^{k} \binom{2k}{k}^2\frac{O_k^2}{(2k-1)^2}
=2\sum_{k>0} \left( \frac{1}{16} \right)^{k} \binom{2k}{k}^2\frac{t_k(1,1)}{(2k-1)^2}
+\sum_{k>0} \left( \frac{1}{16} \right)^{k} \binom{2k}{k}^2\frac{t_k(2)}{(2k-1)^2}.
\end{align*}
 Now, for all $\bfs=(s_1,\dots,s_d) \in \mathbb{N}^d$ and $m \in \mathbb{N}$,
\begin{align*}
&\, \sum_{k>0} \left( \frac{1}{16} \right)^{k} \binom{2k}{k}^2\frac{t_k(\bfs)}{(2k-1)^m} \\
=&\, \sum_{k>k_2>\cdots>k_d>0} \left( \frac{1}{16} \right)^{k} \binom{2k}{k}^2\frac{1}{(2k-1)^{s_1+m}(2k_2-1)^{s_2}\cdots (2k_d-1)^{s_d}} \\
&\, +\sum_{k>k_1>\cdots>k_d>0} \left( \frac{1}{16} \right)^{k} \binom{2k}{k}^2\frac{1}{(2k-1)^{m}(2k_1-1)^{s_1}\cdots (2k_d-1)^{s_d}},
\end{align*}
 and this can be handled by the approach in \cite{XuZhao2022}.
 It is not hard to see that this above procedure can be generalized to treat arbitrary powers of $O_k$, not only its square.

 Our last result, Theorem~\ref{thm-Varz}, differs noticeably from the others  in two aspects: First, the binomial coefficients appear as a first power 
 instead of a square; second, the series involves a variable $z$ so that its specializations may offer many series identities involving log-sine integrals. For 
 example, a harmonic analog of Theorem~\ref{thm-Varz} has been found in Charlton et al. \cite[Lemma 2.1]{CGLXZ2022} who use them to prove the 
 following identity 
\begin{align*}
 \sum_{n=0}^\infty\frac{\binom{2n}{n}}{(2n+1)^3 16^n}\left(9H_{2n+1}+\frac{32}{2n+1}\right)=40\beta(4)+\frac{5}{12}\pi \zeta(3) 
\end{align*}
first conjectured by Z.-W. Sun \cite[Conjecture 10.60]{Sun2021}. It is highly likely that Theorem~\ref{thm-Varz} can be further generalized to evaluate other Ap\'ery-like sums in closed forms.

\subsection*{Competing interests statement}
 There are no competing interests to declare.

\subsection*{Acknowledgements}
Ce Xu is supported by the National Natural Science Foundation of China (Grant No. 12101008), the Natural Science Foundation of Anhui Province (Grant No. 2108085QA01), and the University Natural Science Research Project of Anhui Province. The corresponding Grant Number for the last case is KJ2020A0057. J. Zhao is supported by the Jacobs Prize from The Bishop's School.
 The authors are very grateful for the referee feedback provided, which has substantially improved our article.

\end{document}